\documentclass[a4paper,11pt,reqno]{amsart}
\usepackage[margin=1.4in]{geometry}

\usepackage[style=alphabetic, backend=bibtex, maxbibnames=99]{biblatex}
\AtEveryBibitem{\ifentrytype{article}{\clearfield{url} \clearfield{issn}}{}}

\usepackage[colorlinks=true]{hyperref}
\hypersetup{citecolor=Sepia, linkcolor=Sepia, urlcolor=MidnightBlue}

\usepackage{lmodern}

\usepackage{amsmath}        % extensions for typesetting of math
\usepackage{amsfonts}       % math fonts
\usepackage{amsthm}         % theorems, definitions, etc.
\usepackage{bbding}         % various symbols (squares, asterisks, scissors, ...)
\usepackage{bm}             % boldface symbols (\bm)
\usepackage{graphicx}       % embedding of pictures
\usepackage{fancyvrb}       % improved verbatim environment
\usepackage{dcolumn}        % improved alignment of table columns
\usepackage{booktabs}       % improved horizontal lines in tables
\usepackage{paralist}       % improved enumerate and itemize
\usepackage[usenames, dvipsnames]{xcolor}  % typesetting in color
\usepackage{amssymb}%
\usepackage{faktor}
\usepackage{xfrac}
\usepackage{verbatim}
\usepackage{mathtools}
\usepackage[shortlabels]{enumitem}
	\setlist[enumerate,1]{label=(\roman*), font=\normalfont}

\usepackage{array}   % for \newcolumntype macro
\newcolumntype{L}{>{$}l<{$}} % math-mode version of "l" column type
\newcolumntype{C}{>{$}c<{$}} % math-mode version of "l" column type
\usepackage{ wasysym }
\usepackage{accents}
\usepackage{thmtools}

\usepackage{relsize} %bigger math symbols
\usepackage{multirow}
\usepackage[graphicx]{realboxes}
\usepackage{adjustbox}
\usepackage{pdflscape}
\usepackage{rotating}
\usepackage{multicol}
\usepackage{mdframed}

\usepackage{phaistos}

\theoremstyle{plain}
\newtheorem{theorem}{Theorem}[section]
\newtheorem{lemma}[theorem]{Lemma}
\newtheorem{proposition}[theorem]{Proposition}
\newtheorem{corollary}[theorem]{Corollary}

\newcounter{CountQuestions}

\theoremstyle{definition}
\newtheorem{definition}[theorem]{Definition}
\newtheorem{remark}[theorem]{Remark}
\newtheorem{example}[theorem]{Example}

\makeatletter
\g@addto@macro\@floatboxreset\centering
\makeatother

\newcommand{\goto}{\rightarrow}
\newcommand{\abs}[1]{\left|{#1}\right|}

\def\quotient#1#2{%
    \raise1ex\hbox{$#1$}\Big/\lower1ex\hbox{$#2$}%
}

\DeclareMathOperator{\chrst}{char}
\DeclareMathOperator{\ndeg}{ndeg}

\DeclareMathOperator{\spn}{span}
\newcommand{\Ort}{\mathop{\large{ \mathlarger{\ort}}}}

\newcommand{\uv}[1]{``{#1}"}

\newcommand{\F}{\mathbb{F}}

\newcommand{\mB}{\mathcal{B}}

\newcommand{\simsim}{\stackrel{{\scriptsize{\mathrm{sim}}}}{\sim}}

\newcommand{\iql}[1]{\mathfrak{i}_{\mathrm{d}}(#1)}

\newcommand{\an}{\mathrm{an}}

\newcommand{\SP}[1]{\mathrm{SP}(#1)}

\newcommand{\ssp}[1]{\mathrm{sSP}(#1)}
\newcommand{\fullsplitpat}[1]{\mathrm{fSP}(#1)}
\newcommand{\fsp}[1]{\mathrm{fSP}(#1)}
\newcommand{\pisp}[1]{\mathrm{piSP}(#1)}

\newcommand{\ort}{\:\bot\:}

\newcommand{\sqf}[1]{\langle #1 \rangle} % singular form
\newcommand{\pf}[1]{\langle\!\langle #1 \rangle\!\rangle} % quasi Pfister form
\newcommand{\nf}{\hat{\nu}} %norm form
\begin{document}
\title[Isotropy and full splitting pattern of quasilinear $p$-forms]{Isotropy and full splitting pattern of quasilinear $\boldsymbol{p}$-forms}
\author{Krist\'yna Zemkov\'a}
\address{Fakult\"at f\"ur Mathematik, Technische Universit\"at Dortmund, D-44221 Dortmund,
Germany}
\address{Department of Mathematics and Statistics, University of Victoria, Victoria BC V8W 2Y2, Canada}
\email{zemk.kr@gmail.com}%
\date{\today}

\begin{abstract}
For a quasilinear $p$-form defined over a field $F$ of characteristic ${p>0}$, we prove that its defect over the field ${F(\!\sqrt[p^{n_1}]{a_1}, \dots, \sqrt[p^{n_r}]{a_r})}$ is equal to its defect over the field ${F(\!\sqrt[p]{a_1}, \dots, \sqrt[p]{a_r})}$,  strengthening a result of Hoffmann from 2004. We also compute the full splitting pattern of some families of quasilinear $p$-forms.
\end{abstract}
\thanks{This work was supported by DFG project HO 4784/2-1. The author further acknowledges support from the Pacific Institute for the Mathematical Sciences and a partial support from the National Science and Engineering Research Council of Canada. The research and findings may not reflect those of these institutions.\\ \indent The author reports there are no competing interests to declare.}
\maketitle

%===================================================================
%===============================================================================================
%\section*{Acknowledgment}
%The author acknowledges support from the Pacific Institute for the Mathematical Sciences and a partial support from the National Science and Engineering Research Council of Canada. The research and findings may not reflect those of these institutions. 

%\section*{Declaration of interests}

%The author declares that she has no known competing financial interests or personal relationships that could have appeared to influence the work reported in this paper.
%===============================================================================================
\section{Introduction}

Let $F$ be a field and $\varphi$ a quadratic form over $F$. Then the splitting pattern $\SP{\varphi}$ of $\varphi$ can be defined as the set of dimensions realized by the anisotropic part of $\varphi$ over all possible field extensions of $F$. In case of fields of characteristic different from 2, the splitting pattern was first systematically studied by Knebusch in \cite{Kne76}; Knebusch presented an explicit tower of fields $F_0\subseteq F_1 \subseteq \dots \subseteq F_h$ (depending on $\varphi$), nowadays called the \emph{standard splitting tower}, that has a generic property: $m\in\SP{\varphi}$ if and only if $m$ is the dimension of the anisotropic part of $\varphi$ over one of the fields $F_k$ in the standard splitting tower.

We can also consider an analogous standard splitting tower over fields of characteristic 2. However, in charateristic 2, we have to consider several different types of quadratic forms: \emph{nonsingular}, \emph{totally singular}, and \emph{singular} (the latter being the general case). For nonsingular quadratic forms, the standard splitting tower has been defined by Knebusch \cite{Kne77}, and it has been extended to the general case by Laghribi \cite{Lag02-gensplit}. Unfortunately, the generic behavior of the standard splitting tower is guaranteed only in the case of nonsingular quadratic forms. In fact, Hoffmann and Laghribi \cite{HL04} proved that if $\varphi$ is a nonsingular quadratic form over a field of characteristic $2$, then the standard splitting tower provides all values from the set $\SP{\varphi}$, but the same does not have to hold for singular quadratic forms. More precisely, for (singular) quadratic forms over fields of characteristic 2, we distinguish between \emph{Witt index} of $\varphi$ -- the maximal number of hyperbolic planes split off by $\varphi$ -- and the \emph{defect} of $\varphi$ -- the maximal number of $\sqf{0}$'s split off by (the quasilinear part of) $\varphi$. Hoffmann and Laghribi proved that the standard splitting tower gives all possible Witt indices  (\cite[Prop.~4.6]{HL04}), but not all possible defects  (\cite[Ex.~8.15]{HL04}). Therefore, for a quadratic form $\varphi$ defined over a field of characteristic $2$, we distinguish between the \emph{standard splitting pattern} of $\varphi$ (denoted $\ssp{\varphi}$) -- the set of the dimensions of the anisotropic parts of $\varphi$ over the fields in the standard splitting tower -- and the \emph{full splitting pattern} of $\varphi$ (denoted $\fsp{\varphi}$) -- the set of the dimensions of the anisotropic parts of $\varphi$ over all field extensions. To understand the full splitting pattern of a singular quadratic form, one has to first understand the splitting behavior of its quasilinear part; hence, the starting point is to examine the full splitting pattern of totally singular quadratic forms.

As such, the study of splitting patterns of quadratic forms is closely related to the study of their isotropy behavior over different field extensions. For some results in characteristic 2, see \cite{Lag04, HL06, Hof22, KZ22, LagMuk23}.

The aim of this note is to extend the knowledge on isotropy and full splitting patterns of totally singular quadratic forms. However, we chose to prove our results in a more general setting of quasilinear $p$-forms, as they are a generalization of the concept of totally singular quadratic forms over fields of characteristic $2$: let $p$ be any prime integer, and let $F$ be a field of characteristic $p$. A \emph{quasilinear $p$-form} can be expressed as a \uv{diagonal} homogeneous polynomial of degree $p$ with coefficients in $F$, i.e., as $a_1X_1^p+\dots+a_nX_n^p$ with $a_1,\dots,a_n\in F$ and $X_1,\dots,X_n$ variables; we denote it by $\sqf{a_1,\dots,a_n}$ (in an analogy to the totally singular quadratic forms). Many results (and proofs) on totally singular quadratic forms translate directly to quasilinear $p$-forms, but some need a special care because of the increased complexity of the forms. (For example: 1-fold \emph{quasi-Pfister form} is defined as $\pf{a}=\sqf{1,a,\dots,a^{p-1}}$. These forms are easy to work with when $p=2$, i.e., in the case of totally singular quadratic forms, but it can get more complicated when $p>2$.) Quasilinear $p$-forms have been studied first by Hoffmann \cite{Hof04} and subsequently by Scully \cite{Scu13, Scu16-Hoff, Scu16-Split}. In particular, \cite{Scu16-Split} analyzes the standard splitting pattern of quasilinear $p$-forms. Moreover, note that quasilinear $p$-forms represent $F^p$-subspaces of $F$; as such, they can be used as tools in other areas (see, e.g., \cite{Chap22}). 

\bigskip

In this article, we focus on two topics that are closely related. In the first part of the article,  Section~\ref{Sec:isotropy}, we study the defect of quasilinear $p$-forms over purely inseparable field extensions. The \uv{complexity} of a purely inseparable field extension $E$ of a field $F$ depends on its \emph{exponent} -- the smallest integer $e$ (if it exists) such that $E^{p^e}\subseteq F$. It turns out that to fully describe the defects over any purely inseparable field extensions, we only need to understand the defects over purely inseparable extensions of exponent one. More precisely, we will prove a stronger version of Hoffmann's theorem \cite[Th.~5.9]{Hof04}:

\begin{theorem} [cf. Theorem~{\ref{Th:IsotropyInsepExtKZ_pforms}}]
Let $r\geq1$. For each $1\leq i \leq r$, let $a_i\in F$, $n_i\geq1$, and $\alpha_i$, $\beta_i$ be such that $\alpha_i^p=\beta_i^{p^{n_i}}=a_i$. Furthermore, set $K=F(\alpha_1,\dots,\alpha_r)$,  $L=F(\beta_1,\dots,\beta_r)$ and $\pi=\pf{a_1,\dots,a_r}$. Let $\varphi$ be an anisotropic quasilinear $p$-form. Then
\begin{enumerate}
	\item $\iql{\varphi_K}=\iql{\varphi_L}$,
	\item $\iql{\varphi_K}=\frac{1}{p^r}\iql{\varphi\otimes\pi}$ if $\pi$ is anisotropic.
\end{enumerate}
\end{theorem}

In the second part of the article, Section~\ref{Sec:fsp}, we look directly at the full splitting pattern of quasilinear $p$-forms. There are two obvious bounds on the size of the set $\fsp{\varphi}$: The upper bound is given by the dimension of $\varphi$. On the other hand, the inclusion $\ssp{\varphi}\subseteq\fsp{\varphi}$ gives a lower bound $b:=\abs{\ssp{\varphi}}$ on $\abs{\fsp{\varphi}}$; in Corollary~\ref{Cor:FSPlowerBound}, we provide an alternative proof of $b\leq\abs{\fsp{\varphi}}$ by using purely inseparable field extensions instead of the standard splitting tower. In examples~\ref{Ex:fspMin} and \ref{Ex:fspPF}, we compute the full splitting patterns of the so-called \emph{minimal} $p$-forms and of \emph{quasi-Pfister forms}, showing that both of the above mentioned bounds are optimal. After that we focus on some quasilinear $p$-forms with nontrivial full splitting pattern, namely on a subfamily of quasi-Pfister neighbors:

\begin{theorem} [cf. Theorem~\ref{Th:fullsplitpatSPNmin_pforms}]
Let $\pi$ be an $n$-fold quasi-Pfister form over $F$, $\sigma$ be a minimal subform of $\pi$ of dimension at least $2$ and $d\in F^*$ be such that the quasilinear $p$-form $\varphi\simeq\pi\ort d\sigma$ is anisotropic. Then $m\in \fullsplitpat{\varphi}$ if and only if $m=p^k+\ell$ and one of the following holds:
\begin{enumerate}
	\item $0\leq k\leq n$ and $\ell=0$;
	\item $0\leq k\leq n$ and $\max\{1, k-n+\dim\sigma\}\leq \ell\leq \min\{\dim\sigma,p^k\}$.
\end{enumerate}
\end{theorem}

We conclude the paper by Example~\ref{Ex:FullSplitPattSPN_pforms}, in which we apply the previous theorem on a quasilinear $p$-form of dimension $p^4+4$; we compute its full splitting pattern and explicitly provide all the field extensions that we use for that.

\bigskip

This paper is based on the second chapter of the author's PhD thesis \cite{KZdis}.

%=====================================================================================================

\section{Preliminaries}\label{Sec:Prel_p-forms}

All fields in this article are of characteristic $p>0$.

%--------------------------------------------

\subsection{Quasilinear $p$-forms} Quasilinear $p$-forms are a generalization of totally singular quadratic forms (defined over fields of characteristic $2$). For totally singular quadratic forms, see, e.g., \cite[Sec.~8]{HL04}; for a detailed introduction into quasilinear $p$-forms in general, see \cite{Hof04}.

\begin{definition}
Let $F$ be a field and $V$ a finite-dimensional vector space over $F$. A \emph{quasilinear $p$-form} (or simply a \emph{$p$-form}) over $F$ is a map ${\varphi:V\goto F}$ with the following properties:
\begin{enumerate}[(1)]
	\item $\varphi(av)=a^p\varphi(v)$ for any $a\in F$ and $v\in V$, \label{Enum:Defpform1}
	\item $\varphi(v+w)=\varphi(v)+\varphi(w)$ for any $v,w\in V$. \label{Enum:Defpform2}
\end{enumerate}
We define the dimension of $\varphi$ as $\dim\varphi=\dim V$. 
\end{definition}

\bigskip

Let $V$ be an $F$-vector space with a basis $\{e_1,\dots,e_n\}$, and $\varphi$ a $p$-form on $V$. Let $a_i=\varphi(e_i)$; then, for a vector $v=\sum_{i=1}^nx_ie_i\in V$, we have
\[\varphi(v)=\sum_{i=1}^na_ix_i^p.\]
Thus we can associate $\varphi$ with the \uv{diagonal} homogeneous polynomial $\sum_{i=1}^na_iX_i^p\in F[X_1,\dots,X_n]$. We denote such a $p$-form by $\sqf{a_1,\dots,a_n}$.

\bigskip

Let $\sqf{a_1,\dots,a_n}$ and $\sqf{b_1\dots,b_m}$ be $p$-forms over $F$ . Then  we define the orthogonal sum and the tensor product of these two $p$-forms in the obvious way:
\begin{align*}
\sqf{a_1,\dots,a_n}\ort\sqf{b_1\dots,b_m}&=\sqf{a_1,\dots,a_n,b_1\dots,b_m},\\
\sqf{a_1,\dots,a_n}\otimes\sqf{b_1\dots,b_m}&=a_1\sqf{b_1\dots,b_m}\ort\dots\ort a_n\sqf{b_1\dots,b_m}.
\end{align*}
Moreover, for $c\in F^*$, we have $c\sqf{a_1,\dots,a_n}=\sqf{ca_1,\dots,ca_n}$, and if $k$ is a positive integer, then we write $k\times\varphi$ for the $p$-form $\varphi\ort\dots\ort\varphi$ ($k$ copies).

Let $\varphi:V\goto F$ and $\psi:W\goto F$ be two $p$-forms and $f:V\goto W$ a vector space homomorphism satisfying $\varphi(v)=\psi(f(v))$ for any $v\in V$. If $f$ is bijective, then we call $\varphi$ and $\psi$ \emph{isometric}, and write $\varphi\simeq\psi$. If $f$ is injective, then $\varphi$ is a \emph{subform} of $\psi$, which we denote by $\varphi\subseteq\psi$; in such case, there exists a $p$-form $\tau$ over $F$ such that $\psi\simeq\varphi\ort\tau$. If there exists $c\in F^*$ such that $\varphi\simeq c\psi$, then we call $\varphi$ and $\psi$ \emph{similar}, which we denote by $\varphi\simsim\psi$.

We call a $p$-form  $\varphi:V\goto F$ \emph{isotropic} (or \emph{defective}) if $\varphi(v)=0$ for some nonzero vector $v\in V$; otherwise, $\varphi$ is called \emph{anisotropic} (or \emph{nondefective}). The $p$-form $\varphi$ can be written as ${\varphi\simeq\sigma\ort k\times\sqf{0}}$ for some anisotropic $p$-form $\sigma$ over $F$ and a non-negative integer $k$. Then $\sigma$ is unique up to isometry; we call it the \emph{anisotropic part of $\varphi$}, and denote it by $\varphi_{\an}$. The integer $k$ is called the \emph{defect} of $\varphi$ and is denoted by $\iql{\varphi}$. In particular, we have ${\varphi\simeq\varphi_{\an}\ort\iql{\varphi}\times\sqf{0}}$. 

For a $p$-form  $\varphi:V\goto F$, we denote 
\begin{align*}
&D_F(\varphi)=\{\varphi(v)~|~v\in V\} \quad &\text{and}& \quad &D_F^*(\varphi)=D_F(\varphi)\setminus\{0\},\\
 &G_F^*(\varphi)=\{x\in F^*~|~x\varphi\simeq\varphi\} \quad &\text{and}& \quad &G_F(\varphi)=G_F^*(\varphi)\cup\{0\}.
\end{align*}
Then $D_F(\varphi)$ is the set of all elements of $F$ represented by $\varphi$ (including zero). Note that since $F^p$ is a field, it follows that $D_F(\varphi)$ is an $F^p$-vector space. On the other hand, for any $F^p$-vector space $U$, there exists a unique (up to isometry) anisotropic $p$-form $\sigma$ such that $D_F(\sigma)=U$ (see \cite[Prop.~2.12]{Hof04}). Note that it follows that for any $a,b\in F$ and $x\in F^*$, we have
\[\sqf{a}\simeq\sqf{ax^p}, \quad \sqf{a,b}\simeq\sqf{a,a+b}.\]
In particular, $\sqf{-a}\simeq\sqf{a}$ for any $a\in F$, because $-1=(-1)^p$. Moreover, the set $W=\{v\in V~|~\varphi(v)=0\}$ is also an $F^p$-vector space, because for any $v,w\in W$, we have $\varphi(v+w)=\varphi(v)+\varphi(w)=0$; it follows that $W$ is the unique maximal isotropic subspace of $V$.

It is easy to see that $\varphi$ is anisotropic if and only if $\dim\varphi=\dim_{F^p}D_F(\varphi)$. More precisely, we have the following lemma:

\begin{lemma}[{\cite[Prop.~2.6]{Hof04}}] \label{Lemma:PickAnisp-subform} \label{Lem:p-subform}
Let $\varphi$ be a $p$-form over $F$.
\begin{enumerate}
	\item Let $\{c_1,\dots,c_k\}$ be any $F^p$-basis of the vector space $D_F(\varphi)$. Then we have $\varphi_{\an}\simeq\sqf{c_1,\dots,c_k}$.
	\item If $a_1,\dots,a_m\in D_F(\varphi)$, then $\sqf{a_1,\dots,a_m}_{\an}\subseteq\varphi$.
\end{enumerate}
\end{lemma}

%------------------------------------
\subsection{Field extensions} 

Let $\varphi:V\goto F$ be a $p$-form and $E/F$ a field extension. Then we denote by $\varphi_E$ the $p$-form on the $E$-vector space $V_E=E\otimes V$ defined by $\varphi_E(e\otimes v)=e^p\varphi(v)$ for any $e\in E$ and $v\in V$; we have $\dim\varphi=\dim\varphi_E$. In the following, we write $D_E(\varphi)$  for short instead of $D_E(\varphi_E)$.

\bigskip

\begin{definition}
For any element $a\in F$, denote $\pf{a}=\sqf{1,a,\dots,a^{p-1}}$. Let $a_1,\dots,a_n\in F$; then we set $\pf{a_1,\dots,a_n}=\pf{a_1}\otimes\cdots\otimes\pf{a_n}$ and call it an \emph{$n$-fold quasi-Pfister form}. Moreover, for convenience, we call $\sqf{1}$ the \emph{$0$-fold quasi-Pfister form}.

A $p$-form $\varphi$ over $F$ is called a \emph{quasi-Pfister neighbor} if there exists a quasi-Pfister form $\pi$ over $F$ and $c\in F^*$ such that $c\varphi\subseteq\pi$ and $\dim\varphi>\frac{1}{p}\dim\pi$. 
\end{definition}

Let $\pi$ be a quasi-Pfister form over $F$; then $G_F(\pi)=D_F(\pi)$ (\cite[Prop.~4.6]{Hof04}). If $E/F$ is a field extension, then $(\pi_E)_{\an}$ is a quasi-Pfister form (\cite[Lemma~2.6]{Scu13}). Moreover, if $\varphi$ is a quasi-Pfister neighbor of $\pi$, then it holds that $\varphi_E$ is isotropic if and only if $\pi_E$ is isotropic (\cite[Prop.~4.14]{Hof04}).

%---------------------------------------------
\subsection{Field theory} \label{Subsec:FieldTheory}

 Since $F^p$ is a field, we can compare extensions of $F$ and of $F^p$. We will use the following lemma repeatedly.

\begin{lemma}[{\cite[Lemma~7.5]{Hof04}}] \label{Lemma:AlgebraicToThepPower}
Let $\chrst{F}=p$.
\begin{enumerate}
	\item Let $E/F$ be a field extension and $\alpha\in E$. Then $\alpha$ is algebraic over $F$ if and only if $\alpha^p$ is algebraic over $F^p$. Moreover, it holds that
			\[[F(\alpha):F]=[F^p(\alpha^p):F^p].\]
	\item If $a_1,\dots,a_r\in F$, then
			\[[F(\sqrt[p]{a_1},\dots,\sqrt[p]{a_r}):F]=[F^p(a_1,\dots,a_r):F^p].\]
\end{enumerate}
\end{lemma}

\bigskip

We say about a finite set $\{a_1,\dots,a_n\}\subseteq F$ that it is \emph{$p$-independent over $F$} if $[F^p(a_1,\dots,a_n):F^p]=p^n$; otherwise, we call the set (or the elements of the set) \emph{$p$-dependent over $F$}. We say that a set $S\subseteq F$ is \emph{$p$-independent over $F$} if any finite subset of $S$ is $p$-independent over $F$. 

\begin{lemma}[{\cite[pg.~27]{Pick50}}]\label{Lem:p-independence}
Let $a_1,\dots,a_n\in F$. The following are equivalent:
\begin{enumerate}
	\item The set $\{a_1,\dots,a_n\}$ is $p$-independent over $F$,
	\item for any $1\leq i\leq n$, we have $a_i\notin F^p(a_1,\dots,a_{i-1}, a_{i+1},\dots,a_n)$,
	\item the system $\bigl(a_1^{i_1}\cdots a_n^{i_n}~\bigl|~(i_1,\dots,i_n)\in\{0,\dots,p-1\}^n\bigr)$ is $F^p$-linearly independent.
\end{enumerate}
\end{lemma}

\begin{example}
(i) It is clear directly from the definition that an one-element set $\{a\}\subseteq F$ is $p$-independent if and only if $a\notin F^p$.

(ii) Let $a,b\in F\setminus F^p$. Then, by Lemma~\ref{Lem:p-independence}, $a,b$ are $p$-dependent if and only if $a\in F^p(b)$ if and only if $b\in F^p(a)$.
\end{example}

\begin{corollary} \label{Cor:pIndAndAnisPF}
Let $a_1,\dots,a_n\in F^*$. The set $\{a_1,\dots,a_n\}$ is $p$-independent over $F$ if and only if the quasi-Pfister form $\pf{a_1,\dots,a_n}$ is anisotropic over $F$.
\end{corollary}

\begin{proof}
First of all, note that 
\[\pf{a_1,\dots,a_n}\simeq \Ort_{(i_1,\dots,i_n)\in\{0,\dots,p-1\}^n}\sqf{a_1^{i_1}\cdots a_n^{i_n}}.\]
Now $\{a_1,\dots,a_n\}$ is $p$-independent over $F$ if and only if (by Lemma~\ref{Lem:p-independence}) the system $(a_1^{i_1}\cdots a_n^{i_n}~|~(i_1,\dots,i_n)\in\{0,\dots,p-1\}^n)$ is $F^p$-linearly independent if and only if $\dim_{F^p}D_F(\pf{a_1,\dots,a_n})=p^n$ if and only if $\pf{a_1,\dots,a_n}$ is anisotropic over $F$.
\end{proof}

Let $E$ be a field with $F^p\subseteq E\subseteq F$. Note that the extension $E/F^p$ is purely inseparable of exponent one (an \emph{exponent} of a purely inseparable field extension $L/K$ is the infimum of integers $n$ such that $L^{p^n}\subseteq K$). We call a set $\mB\subseteq E$ a \emph{$p$-basis of $E$ over $F$} if $\mB$ is $p$-independent over $F$ and $F^p(\mB)=E$. 

\begin{lemma}[cf. {\cite[Cor.~A.8.9]{CSAandGC}}]\label{Lemma:p-bases} 
Let $E$ be a field with $F^p\subseteq E\subseteq F$. Then the following hold:
\begin{enumerate}[(i)]
	\item There exists a $p$-basis of $E$ over $F$. 
	\item If $\{a_1,\dots,a_n\}\subseteq E$ is $p$-independent over $F$, then there exists a set $\mathcal{A}\subseteq E$ which is $p$-independent over $F$ and such that $\{a_1,\dots,a_n\}\cup\mathcal{A}$ is a $p$-basis of $E$ over $F$.
	\item Let $\mathcal{A}\subseteq E$ be such that $F^p(\mathcal{A})=E$. Then there exists $\mB\subseteq\mathcal{A}$ which is a $p$-basis of $E$ over $F$.
\end{enumerate}
\end{lemma}

\begin{proof}
The proof is an easy exercise on Zorn's lemma: 
We seek a maximal element of $\mathcal{S}$, where:
\begin{align*}
\mathrm{(i)} \quad &\mathcal{S}=\left\{A~|~A\subseteq E,\ A\ p\text{-independent over } F\right\},\\
\mathrm{(ii)} \quad&\mathcal{S}=\left\{A~|~\{a_1,\dots,a_n\}\subseteq A\subseteq E,\ A\  p\text{-independent over } F\right\},\\
\mathrm{(iii)} \quad&\mathcal{S}=\left\{A~|~A\subseteq \mathcal{A}, \ A\ p\text{-independent over } F\right\}. \qedhere
\end{align*}
\end{proof}

%-----------------------------------------
\subsection{Norm fields and norm forms}

Let $\varphi$ be a $p$-form over $F$. We define the \emph{norm field} of $\varphi$ over $F$ as the field 
\[N_F(\varphi)=F^p\left(\frac{a}{b}~\Bigl|~a,b\in D_F^*(\varphi)\right).\] 
Note that $N_F(\varphi)$ is a finite field extension of $F^p$; we define the \emph{norm degree} of $\varphi$ over $F$ as $\ndeg_F{\varphi}=[N_F(\varphi):F^p]$. 

It is obvious from the definition that $N_F(\varphi)=N_F(\varphi_\an)$ and also that $N_F(\varphi)=N_F(c\varphi)$ for any $c\in F^*$. Moreover, if $\psi$ is another $p$-form over $F$ such that $\psi\simeq\varphi$, then $N_F(\psi)=N_F(\varphi)$. Finally, if $\tau_\an\subseteq c\varphi_\an$ for some $p$-form $\tau$ over $F$ and $c\in F^*$, then $N_F(\tau)\subseteq N_F(\varphi)$.

\begin{lemma}[{\cite[Prop.~4.8]{Hof04}}]\label{Lemma:PrereqForNormForm}
Let $\varphi$ be a nonzero $p$-form with $\ndeg_F\varphi=p^n$. Then $n+1\leq\dim\varphi_\an\leq p^n$.
\end{lemma}

Note that the norm degree is always a $p$-power. Let us assume that $N_F(\varphi)=F^p(b_1,\dots,b_n)$ for some $b_i\in F$; then $\{b_1,\dots,b_n\}$ is $p$-independent over $F$ if and only if $\ndeg_F{\varphi}=p^n$. By Lemma~\ref{Lemma:p-bases}, such a $p$-basis of $N_F(\varphi)$ over $F$ always exists.

\begin{lemma}[{\cite[Lemma~4.2 and Cor.~4.3]{Hof04}}]\label{Lemma:NormFieldBasics}
Let $\varphi$ be a $p$-form over $F$.
\begin{enumerate}
	\item If $\varphi\simeq\sqf{a_0,\dots,a_n}$ for some $n\geq1$ and $a_1,\dots,a_n\in F$ with $a_0\neq0$, then $N_F(\varphi)=F^p\bigl(\frac{a_1}{a_0},\dots,\frac{a_n}{a_0}\bigr)$.
	\item Suppose $N_F(\varphi)=F^p(b_1,\dots,b_m)$ for some $b_1,\dots,b_m\in F^*$  and let $E/F$ be a field extension. Then $N_E(\varphi)=E^p(b_1,\dots,b_m)$.
\end{enumerate}
\end{lemma}

Recall that, by Lemma~\ref{Lemma:p-bases}, any $p$-generating set contains a $p$-basis. Therefore, part (i) of the previous lemma implies the following:

\begin{corollary}\label{Cor:BasisSubform}
Let $\varphi\simeq\sqf{a_0,\dots,a_n}$ for $n\geq1$ and $a_0,\dots,a_n\in F$ with $a_0\neq0$. Moreover, suppose that $\ndeg_F\varphi=p^k$. Then there exists a subset $\{i_1,\dots,i_k\}\subseteq\{1,\dots,n\}$ such that $\left\{\frac{a_{i_1}}{a_0},\dots,\frac{a_{i_k}}{a_0}\right\}$ is a $p$-basis of $N_F(\varphi)$ over $F$.

 In particular, if $1\in D_F^*(\varphi)$, then there exist $b_1,\dots,b_n\in F$ such that $\varphi\simeq\sqf{1,b_1,\dots,b_n}$ and $N_F(\varphi)=F^p(b_1,\dots,b_k)$.
\end{corollary}

Let $\varphi$ be a $p$-form over $F$ with $\ndeg_F{\varphi}=p^n$ and ${N_F(\varphi)=F^p(a_1,\dots,a_n)}$ (so, in particular, $\{a_1,\dots,a_n\}$ is $p$-independent over $F$). Then we define the \emph{norm form} of $\varphi$ over $F$, denoted by $\nf_F(\varphi)$, as the quasi-Pfister form $\pf{a_1,\dots,a_n}$. It follows that $\nf_F(\varphi)$ is the smallest quasi-Pfister form that contains a scalar multiple of $\varphi_\an$ as its subform. In particular, if $\varphi$ is a quasi-Pfister form itself, then we have $\nf_F(\varphi)\simeq\varphi_\an$.

\bigskip

We state a condition under which a tensor product of two anisotropic $p$-forms remains anisotropic.

\begin{lemma} \label{Lemma:AnisotropicProduct_pforms}
Let $\varphi$, $\psi$ be two anisotropic $p$-forms over $F$, and assume that ${\ndeg_F(\varphi\otimes\psi)=\ndeg_F\varphi\cdot\ndeg_F\psi}$. Then $\varphi\otimes\psi$ is anisotropic.
\end{lemma}

\begin{proof} 
We have $\varphi\otimes\psi\subseteq\nf_F(\varphi)\otimes\nf_F(\psi)$, where $\nf_F(\varphi)\otimes\nf_F(\psi)$ is a quasi-Pfister form. As $\nf_F(\varphi\otimes\psi)$ is the smallest quasi-Pfister form containing $\varphi\otimes\psi$, we have $\nf_F(\varphi\otimes\psi)\subseteq\nf_F(\varphi)\otimes\nf_F(\psi)$. By the assumption, we have $\dim\nf_F(\varphi\otimes\psi)=\dim\nf_F(\varphi)\cdot\dim\nf_F(\psi)$; thus, $\nf_F(\varphi\otimes\psi)\simeq\nf_F(\varphi)\otimes\nf_F(\psi)$, and this form is anisotropic. Therefore, its subform $\varphi\otimes\psi$ is anisotropic as well.
\end{proof}

\begin{remark}
The other implication in the previous lemma does not hold in general: Let $\{a,b,c\}\subseteq F^*$ be a $p$-independent set over $F$, and consider the $p$-forms $\varphi\simeq\sqf{1,a,b,c}$ and $\psi\simeq\sqf{1,abc}$. Then 
\[\varphi\otimes\psi\simeq\sqf{1,a,b,c,abc, a^2bc, ab^2c, abc^2}\subseteq\pf{a,b,c},\]
and hence $\varphi\otimes\psi$ is anisotropic. But 
\[\ndeg_F(\varphi\otimes\psi)=p^3<p^3\cdot p=\ndeg_F\varphi\cdot\ndeg_F\psi.\]
\end{remark}

%-----------------------------------------------------------------------------------------
\subsection[Minimal $p$-forms]{Minimal ${p}$-forms} \label{Subsec:Minimalpforms}

In this subsection, we introduce a new class of $p$-forms, so-called minimal $p$-forms. These are the $p$-forms of minimal dimension with respect to their norm degree (cf. Lemma~\ref{Lemma:PrereqForNormForm}):

\begin{definition}
Let $\varphi$ be an anisotropic $p$-form over $F$. We call $\varphi$ \emph{minimal} over $F$ if $\ndeg_F\varphi=p^{\dim\varphi-1}$.
\end{definition}

First, we state some rather obvious properties of minimal $p$-forms.

\begin{lemma}\label{Lem:PropertiesOfMinimalForms_pforms}
Let $\varphi$ and $\psi$ be $p$-forms over $F$ of dimension at least two.
\begin{enumerate}
	\item The $p$-form $\sqf{1,a_1,\dots,a_n}$ is minimal over $F$ if and only if the set $\{a_1,\dots,a_n\}$ is $p$-independent over $F$.
	%\item Minimality is not invariant under field extensions.
	\item If $\varphi$ is minimal over $F$ and $\psi\simsim\varphi$, then $\psi$ is minimal over $F$.
	\item If $\varphi$ is minimal over $F$ and $\psi\subseteq c\varphi$ for some $c\in F^*$, then $\psi$ is minimal over $F$.
	\item If $\ndeg_F\varphi=p^k$ with $k\geq1$, then $\varphi$ contains a minimal subform of dimension $k+1$.
\end{enumerate}
\end{lemma}

\begin{proof}
(i) $\varphi=\sqf{1,a_1,\dots,a_n}$ is minimal over $F$ if and only if $\ndeg_F(\varphi)=p^{n}$ if and only if $[F^p(a_1,\dots,a_n):F^p]=p^n$ if and only if $\{a_1,\dots,a_n\}$ is $p$-independent over $F$ (by Lemma~\ref{Lemma:NormFieldBasics}).

%(ii) Let $\{a,b,c\}\subseteq F$ be a $p$-independent set over $F$. Let $E=F(\sqrt[p]{c})$ and consider $\varphi=\sqf{1,a,b,abc}$. Then $N_F(\varphi)=F^p(a,b,abc)=F^p(a,b,c)$, and hence $\varphi$ is minimal over $F$. Nevertheless, $\varphi_E\simeq\sqf{1,a,b,ab}$ is anisotropic but not minimal over $E$, because $\ndeg_E\varphi=p^2$.

(ii) This is a consequence of the fact that similar $p$-forms have the same dimension and the same norm field.

(iii) Invoking (ii), we can assume $\psi\subseteq\varphi$ and $1\in D_F(\psi)$. So, there exist $a_1,\dots a_n\in F^*$ and some $m\leq n$ such that $\psi\simeq\sqf{1,a_1,\dots,a_m}$ and ${\varphi\simeq\sqf{1,a_1,\dots,a_n}}$; see Lemma~\ref{Lem:p-subform}. By (i), $\{a_1,\dots,a_n\}$ is $p$-independent over $F$; thus, $\{a_1,\dots,a_m\}$ has to be $p$-independent over $F$, and the claim follows by applying (i) again.

(iv) By Lemma~\ref{Lemma:PrereqForNormForm}, it follows from the assumption on $k$ that $\dim\varphi_\an\geq2$.  There exists $c\in F^*$ such that $1\in D_F^*(c\varphi_\an)$; then, by Corollary~\ref{Cor:BasisSubform}, $c\varphi_\an\simeq\sqf{1,b_1,\dots,b_n}$ for some $b_1,\dots,b_n\in F^*$ with $n\geq k$ and  $\{b_1,\dots,b_k\}$ $p$-independent over $F$. It follows that $c^{-1}\sqf{1,b_1,\dots,b_k}$ is a subform of $\varphi$ which is minimal over $F$ by parts (i) and (ii).
\end{proof}

\begin{remark}
Note that minimality is not invariant under field extensions: let $\{a,b,c\}\subseteq F$ be a $p$-independent set over $F$. Let $E=F(\sqrt[p]{c})$ and consider $\varphi=\sqf{1,a,b,abc}$. Then $N_F(\varphi)=F^p(a,b,abc)=F^p(a,b,c)$, and hence $\varphi$ is minimal over $F$. Nevertheless, $\varphi_E\simeq\sqf{1,a,b,ab}$ is anisotropic but not minimal over $E$, because $\ndeg_E\varphi=p^2$.
\end{remark}

%=====================================================================================================
\section{Isotropy over purely inseparable field extensions} \label{Sec:isotropy}

Anisotropic $p$-forms remain anisotropic over purely transcendental and separable field extensions by~\cite[Prop.~5.3]{Hof04}. But the situation gets more interesting when we consider purely inseparable field extensions. The easiest one is a simple purely inseparable extension of exponent one; such an extension is of the form $F(\sqrt[p]{a})/F$ for some $a\in F\setminus F^p$. In the following lemma, we take a closer look at the behavior of $p$-forms over such extensions.

\begin{lemma}[{\cite[Lemma~2.27]{Scu16-Hoff}}] \label{Lemma:p-forms_BasicIsotropy}
Let $\varphi$ be a $p$-form over $F$ and let $a\in F\setminus F^p$. Then:
\begin{enumerate}
	\item $D_{F(\sqrt[p]{a})}(\varphi)=D_F(\pf{a}\otimes\varphi)=\sum_{i=0}^{p-1}a^iD_F(\varphi)$. \label{Lemma:p-forms_BasicIsotropy_i}
	\item $\iql{\varphi_{F(\sqrt[p]{a})}}=\frac1p\iql{\pf{a}\otimes\varphi}$.
\end{enumerate}
\end{lemma}

For purely inseparable field extensions of higher exponents, the following is known. 

\begin{proposition}[{\cite[Prop.~5.7]{Hof04}}] \label{Prop:IsotropyInsepExtHof_pforms}
Let $r\geq1$. For each $1\leq i \leq r$, let $a_i\in F$, $n_i\geq1$, and $\alpha_i$, $\beta_i$ be such that $\alpha_i^p=\beta_i^{p^{n_i}}=a_i$. Furthermore, set $K=F(\alpha_1,\dots,\alpha_r)$, $L=F(\beta_1,\dots,\beta_r)$, and assume that $[K:F]=p^r$. Then $[L:F]=p^{n_1+\dots+n_r}$ and the quasi-Pfister form $\pi=\pf{a_1,\dots,a_r}$ is anisotropic. 

Let $\varphi$ be an anisotropic $p$-form. Then the following statements are equivalent.
\begin{enumerate}
\item $\varphi\otimes\pi$ is isotropic,
\item $\varphi_K$ is isotropic,
\item $\varphi_L$ is isotropic.
\end{enumerate}
\end{proposition}

Recently, in \cite[Cor.~3.3]{LagMuk23}, it has been proved that in the situation from Proposition~\ref{Prop:IsotropyInsepExtHof_pforms} in case of $p=2$, not only isotropy, but also quasi-hyperbolicity of totally singular quadratic forms is preserved. (A totally singular quadratic form $\varphi$ is called \emph{quasi-hyperbolic} if $\iql{\varphi}\geq\frac{\dim\varphi}{2}$.)

Note that Proposition~\ref{Prop:IsotropyInsepExtHof_pforms} does not cover all purely inseparable extensions of exponent higher than one; only those extensions $L/F$ are considered, for which $L=F(\!\sqrt[p^{n_1}]{a_1},\dots,\sqrt[p^{n_r}]{a_r})$ with $\pf{a_1,\dots,a_r}$ anisotropic (because the set $\{a_1\dots,a_r\}$ is $p$-independent; this follows from the assumption $[K:F]=p^r$). Purely inseparable field extensions for which such elements $a_1,\dots,a_r$ can be found are called \emph{modular}. A known example of a non-modular extension is $E/F$ with $F=\F_2(a,b,c)$ for some algebraically independent elements  $a,b,c$, and $E=F(\sqrt[4]{a},\sqrt[4]{b^2a+c^2})$, see \cite{Nonmodular}.

However, we can strengthen the results by comparing the values $\iql{\varphi\otimes\pi}$, $\iql{\varphi_K}$ and $\iql{\varphi_L}$; it turns out that we can prove $\iql{\varphi_K}=\iql{\varphi_L}$ even if the extension $L/F$ is not modular.

\begin{theorem} \label{Th:IsotropyInsepExtKZ_pforms}
Let $r\geq1$. For each $1\leq i \leq r$, let $a_i\in F$, $n_i\geq1$, and $\alpha_i$, $\beta_i$ be such that $\alpha_i^p=\beta_i^{p^{n_i}}=a_i$. Furthermore, set $K=F(\alpha_1,\dots,\alpha_r)$,  $L=F(\beta_1,\dots,\beta_r)$ and $\pi=\pf{a_1,\dots,a_r}$. Let $\varphi$ be an anisotropic $p$-form. Then
\begin{enumerate}
	\item $\iql{\varphi_K}=\iql{\varphi_L}$,
	\item $\iql{\varphi_K}=\frac{1}{p^r}\iql{\varphi\otimes\pi}$ if $\pi$ is anisotropic.
\end{enumerate}
\end{theorem}

\begin{proof}
To prove (i), we proceed by induction on $r$. Let $r=1$, and write $a=a_1$, $n=n_1$, $\alpha=\alpha_1$ and $\beta=\beta_1$. Let $V$ be the vector space associated with $\varphi$, and write $V_K=K\otimes_F V$ and $V_L=L\otimes_F V\simeq L\otimes_K V_K$ as usual. 
Let $W_K=\{w\in V_K~|~\varphi_K(w)=0\}$ (resp. $W_L=\{w\in V_L~|~\varphi_L(w)=0\}$) be the maximal isotropic subspace of $V_K$ (resp. $V_L$); then $\iql{\varphi_K}=\dim_KW_K$ (resp. $\iql{\varphi_L}=\dim_LW_L$). We will show that $W_L= L\otimes_K W_K$, which will not only justify the notation but also prove the claim because we have $\dim_L(L\otimes_KW_K)=\dim_KW_K$.

Let $w=\sum_{i=1}^k\gamma_i\otimes w_i\in L\otimes_KW_K$ for some $\gamma_i\in L$ and $w_i\in W_K$; then
\[\varphi_L(w)=\varphi_L\left(\sum_{i=1}^k\gamma_i\otimes w_i\right)=\sum_{i=1}^k\gamma_i^p\varphi_K(w_i)=0.\] 
Hence, $w\in W_L$, and so we get $L\otimes_KW_K\subseteq W_L$.

For a proof of the opposite inclusion, first note that $\varphi_K(v)\in F$ for any $v\in V_K$ by Lemma~\ref{Lemma:p-forms_BasicIsotropy}\ref{Lemma:p-forms_BasicIsotropy_i}.
Now, let $w\in W_L$. Since $\{1,\beta,\beta^2,\dots,\beta^{p^{n-1}-1}\}$ is a basis of $L$ over $K$, we can write 
\[w=\sum_{i=0}^{p^{n-1}-1}\beta^i\otimes w_i\] 
with $w_i\in V_K$. Then
\[0=\varphi_L(w)=\sum_{i=0}^{p^{n-1}-1}\beta^{pi}\varphi_K(w_i).\]
Note that the set $B=\{1,\beta^p, \beta^{2p}, \dots,\beta^{p^n-p}\}$ is a subset of $\{1,\beta, \dots,\beta^{p^n-1}\}$, which is a basis of $L$ over $F$; therefore, $B$ is linearly independent over $F$. Since, as we noted, $\varphi_K(w_i)\in F$ for all the $i$'s, it follows that $\varphi_K(w_i)=0$. Thus, we have $w_i\in W_K$ for all $i$, and hence $w\in L\otimes_KW_K$. It follows that $W_L\subseteq L\otimes_KW_K$, which concludes the proof in the case $r=1$.

\smallskip

Let $r>1$; for $0\leq i \leq r$, set 
\[L_i=F(\beta_1,\dots,\beta_i,\alpha_{i+1},\dots,\alpha_r).\]
Then $L_0=K$, $L_r=L$, and $\iql{\varphi_{L_i}}=\iql{\varphi_{L_{i+1}}}$ for all $0\leq i\leq r-1$ by the previous part of the proof. Therefore, $\iql{\varphi_K}=\iql{\varphi_L}$.

\smallskip

(ii) Set $K_i=F(\alpha_1,\dots,\alpha_i)$ for $0\leq i\leq r$. Since $\pi$ is anisotropic by the assumption, we have $[K_{i+1}:K_i]=p$ for all $0\leq i\leq r-1$. Now the equality $\iql{\varphi\otimes\pi}=\frac{1}{p^r}\iql{\varphi_K}$ follows by a repeated application of Lemma~\ref{Lemma:p-forms_BasicIsotropy}:
\begin{multline*}
\iql{\varphi_K}=\iql{\varphi_{K_r}}=p\,\iql{(\pf{a_r}\otimes\varphi)_{K_{r-1}}}=\dots\\
=p^r\iql{(\pf{a_1,\dots,a_r}\otimes\varphi)_{K_{0}}}=p^r\iql{\varphi\otimes\pi}. \qedhere
\end{multline*}
\end{proof}

Next to the standard and full splitting pattern, we can define \emph{purely inseparable splitting pattern} of a $p$-form $\varphi$ over $F$ as
\[\pisp{\varphi}=\{\dim(\varphi_L)_\an~|~L/F \text{ a purely inseparable extension}\}.\]
By the previous theorem, we only need to consider purely inseparable extensions of exponent one.

\begin{corollary}\label{Cor:SimplifyFSP_pforms}
Let $\varphi$ be a $p$-form over $F$. Then
\[\pisp{\varphi}=\{\dim(\varphi_K)_{\an}~|~ K/F \text{ finite purely inseparable of exponent } 1\}.\]
\end{corollary}

\begin{remark}
We can restrict the field extensions necessary to determine $\pisp{\varphi}$ a bit more: Let $\varphi$ be anisotropic with $\ndeg_F\varphi=p^n$ and $N_F(\varphi)=F^p(a_1,\dots,a_n)$. Consider the field $K=F(\sqrt[p]{b_1},\dots,\sqrt[p]{b_r})$, write $\pi\simeq\pf{b_1,\dots,b_r}$, and assume that $\pi$ is anisotropic. Then we know from Proposition~\ref{Prop:IsotropyInsepExtHof_pforms} that $\varphi_K$ is isotropic if and only if $\varphi\otimes\pi$ is isotropic. By Lemma~\ref{Lemma:AnisotropicProduct_pforms}, this can happen only if $\ndeg_F(\varphi\otimes\pi)<\ndeg_F(\varphi)\cdot\ndeg_F(\pi)$, i.e., if the set $\{a_1,\dots,a_n,b_1,\dots,b_r\}$ is $p$-dependent. Thus, we have
\[\pisp{\varphi}=\{\dim(\varphi_K)_\an~|~ [K^p(a_1,\dots,a_n):F^p]<p^n[K^p:F^p]\},\]
where $K$ runs over purely inseparable extensions of $F$ of exponent one.

This also indicates that a $p$-form may become isotropic even over a field $E$ such that $E^p\cap N_F(\varphi)=F^p$. For example, consider a set $\{a_1,a_2,a_3,b_1\}$ $p$-independent over $F$, and write $\varphi\simeq\sqf{1,a_1,a_2,a_3}$ and $E=F(\sqrt[p]{b_1},\sqrt[p]{b_2})$ with $b_2=\frac{a_1b_1+a_3}{a_2}$. It can be shown that indeed $E^p\cap N_F(\varphi)=F^p$. Moreover, $\sqf{a_1b_1,a_2b_2,a_3}\subseteq\varphi\otimes\pf{b_1,b_2}$,
and since $a_3=a_1b_1+a_2b_2$, the form $\sqf{a_1b_1,a_2b_2,a_3}$ is isotropic. Therefore, the $p$-form $\varphi\otimes\pf{b_1,b_2}$ is isotropic, and so is $\varphi_E$.
\end{remark}

%==========================================================================================================
\section{Full splitting pattern} \label{Sec:fsp}

We define the \emph{full splitting pattern} of a $p$-form $\varphi$ over $F$ as
\[\fullsplitpat{\varphi}=\{\dim(\varphi_E)_{\an}~|~E/F \text{ a field extension}\}.\]
The goal of this section is to determine the full splitting pattern at least of some families of $p$-forms.

%-----------------------------------------------------------------
\subsection[Bounds on the size of \texorpdfstring{$\fsp{\varphi}$}{fsp(varphi)}]{Bounds on the size of ${\fsp{\varphi}}$}

Putting aside the trivial case when $\varphi_\an$ is the zero form, it is obvious that $\fsp{\varphi}\subseteq\{1,2,\dots,\dim\varphi\}$, and hence $\abs{\fsp{\varphi}}\leq\dim\varphi$. We will show that the bound cannot be improved. Moreover, we provide a lower bound for $\abs{\pisp{\varphi}}$; since obviously $\pisp{\varphi}\subseteq\fsp{\varphi}$, we will also get a lower bound for $\abs{\fsp{\varphi}}$.

\bigskip

For a given $p$-form $\varphi$, we construct a tower of fields over which the defect of $\varphi$ is strictly increasing similarly as over the standard splitting tower. To do that, we put a seemingly strong assumption on the $p$-form. But by Corollary~\ref{Cor:BasisSubform}, all $p$-forms which represent one fulfill this assumption.

\begin{lemma} \label{Lemma:TotInsepTower}
Let $\varphi=\sqf{1,a_1,\dots,a_n}$ be a $p$-form over $F$, and assume that the set $\{a_1, \dots, a_m\}$ is a $p$-basis of the field $N_F(\varphi)$ over $F$ for some $m\leq n$. We denote $E_0=F$ and $E_i=F\bigl(\sqrt[p]{a_1}, \dots, \sqrt[p]{a_i}\bigr)$ for $1\leq i \leq m$. Then we have $\iql{\varphi_{E_{i+1}}}>\iql{\varphi_{E_i}}$ and $\sqf{1,a_{i+1},\dots,a_m}_{E_i}\subseteq(\varphi_{E_i})_{\an}$ for any $0\leq i \leq m-1$. 
\end{lemma}

\begin{proof}

Let $i\in\{0,\dots,m-1\}$ and set $\tau_i= (\varphi_{E_i})_{\an}$. Since $D_F(\varphi)\subseteq D_{E_i}(\tau_i)$, it follows from Lemma \ref{Lem:p-subform} that 
\[(\sqf{1, a_{i+1},\dots,a_m}_{E_i})_{\an}\subseteq \tau_i,\]
and so we only need to prove that the $p$-form $\sqf{1,a_{i+1},\dots,a_m}_{E_i}$ is anisotropic over $E_i$. Suppose not; then we can find $k\in\{i+1,\dots,m\}$, such that 
\[a_k\in\spn_{E_i^p}\{1,a_{i+1}, \dots, \widehat{a_k}, \dots, a_m\}.\] 
But that implies $a_k\in F^p(a_1,\dots,\widehat{a_k}, \dots, a_m)$, which is a contradiction with the choice of $\{a_1,\dots, a_m\}$ as a $p$-basis over $F$ by Lemma~\ref{Lem:p-independence}. 

Furthermore, the $p$-form $\sqf{1, a_{i+1},\dots,a_m}$ becomes isotropic over the field $E_{i+1}$, and hence so does $\tau_i$. It follows that $\iql{\varphi_{E_{i+1}}}>\iql{\varphi_{E_i}}$.
\end{proof}

As a corollary to the previous lemma, we get a lower bound on the size of the sets $\pisp{\varphi}$ and $\fsp{\varphi}$. 

\begin{corollary} \label{Cor:FSPlowerBound}
Let $\varphi$ be a $p$-form over $F$, and suppose that $\ndeg_F\varphi=p^m$. Then $\abs{\pisp{\varphi}}\geq m+1$ and $\abs{\fullsplitpat{\varphi}}\geq m+1$. 
\end{corollary}

\begin{proof}
Since the splitting patterns do not depend on the choice of the similarity class, we can assume without loss of generality that $1\in D_F(\varphi)$; Then, invoking Corollary~\ref{Cor:BasisSubform}, we can suppose that $\varphi$ is as in Lemma~\ref{Lemma:TotInsepTower}. Using the fields $E_i$ for $0\leq i\leq m$ from that lemma, it follows that we have $\dim(\varphi_{E_i})_{\an}\neq\dim(\varphi_{E_j})_{\an}$ for any $i\neq j$. The claim follows. 
\end{proof}

\begin{remark}
(i) The same lower bound on $\abs{\fullsplitpat{\varphi}}$ as in Corollary~\ref{Cor:FSPlowerBound} could be obtained through the standard splitting tower: Denote
\[\ssp{\varphi}=\{\dim(\varphi_K)_{\an}~|~K \text{ a field in the standard splitting tower of }\varphi\}.\]
By \cite[Lemma~4.3]{Scu16-Split}, it holds that $\ndeg_{F(\varphi)}(\varphi_{F(\varphi)})_\an=\frac{1}{p}\ndeg_F\varphi$; therefore, it follows that $\abs{\ssp{\varphi}}=m+1$ where $\ndeg_F\varphi=p^m$. Since $\ssp{\varphi}\subseteq\fsp{\varphi}$, we get $\abs{\fsp{\varphi}}\geq m+1$.

(ii) The tower of fields $F=E_0\subseteq E_1 \subseteq\dots\subseteq E_m$ in Lemma~\ref{Lemma:TotInsepTower} might look like a purely inseparable analogue of the standard splitting tower (note that they are of the same length). But the values of $\iql{\varphi_{E_i}}$ depend on the ordering of $a_1,\dots,a_m$, so in particular we cannot expect that ${\iql{\varphi_{E_1}}=\iql{\varphi_{F(\varphi)}}}$.

For example, if $\varphi\simeq\pf{a_1,a_2}\ort a_3\sqf{1,a_1}$ for some set $\{a_1,a_2,a_3\}\subseteq F^*$ that is $p$-independent over $F$, then $(\varphi_{F(\sqrt[p]{a_1})})_{\an}\simeq(\pf{a_2}\ort a_3\sqf{1})_{F(\sqrt[p]{a_1})}$. The $p$-form $(\varphi_{F(\varphi)})_{\an}\simeq\pf{a_1,a_2}_{F(\varphi)}$ corresponds rather to $(\varphi_{F(\sqrt[p]{a_3})})_{\an}\simeq\pf{a_1,a_2}_{F(\sqrt[p]{a_3})}$.
\end{remark}

The following example shows that the lower bound from Corollary~\ref{Cor:FSPlowerBound} is optimal.

\begin{example} \label{Ex:fspMin}
Let $\varphi\simeq\sqf{1,a_1,\dots,a_m}$ be a minimal $p$-form over $F$ (i.e., the set $\{a_1,\dots,a_m\}$ is $p$-independent over $F$). Then $\ndeg_F{\varphi}=p^m$ and by Lemma~\ref{Lemma:TotInsepTower}, we have
\[\fullsplitpat{\varphi}=\{1,2,\dots,m+1\}. \]
\end{example}

%------------------------------------------------------------------------------
\subsection{Full splitting pattern of quasi-Pfister forms}

To determine the full splitting pattern of quasi-Pfister forms, we start by looking at their behavior over purely inseparable field extensions.

\begin{lemma}
\label{Lem:PFinsep1}
Let $\pi$ and $\pi'$ be anisotropic quasi-Pfister forms over $F$ such that $\pi\simeq\pf{a}\otimes\pi'$ for some $a\in F$, and let $E=F(\sqrt[p]{a})$. Then $\pi'_E$ is anisotropic. 
\end{lemma}

\begin{proof}
Let $\pi'=\pf{b_1,\dots,b_m}$. If $\pi'_E$ is isotropic, then, by Corollary~\ref{Cor:pIndAndAnisPF}, $\{b_1,\dots,b_m\}$ is $p$-dependent over $E$; it follows by Lemma~\ref{Lem:p-independence} that we can find $i\in\{1,\dots,m\}$ such that ${b_i\in E^p(b_1,\dots,b_{i-1},b_{i+1},\dots,b_m)}$. Then $b_i\in F^p(a, b_1,\dots,b_{i-1},b_{i+1},\dots,b_m)$, so the set $\{a, b_1,\dots,b_n\}$ is $p$-dependent over $F$, and hence $\pi$ is isotropic by Corollary~\ref{Cor:pIndAndAnisPF}.
\end{proof}

By induction, we get:

\begin{lemma}\label{Lemma:PFinsepsplitting_pforms} 
Let $\pi=\pf{a_1,\dots,a_m}$ be an anisotropic quasi-Pfister form over $F$. We set $E_0=F$ and $E_i=F\bigl(\sqrt[p]{a_1}, \dots, \sqrt[p]{a_i}\bigr)$ for $1\leq i\leq m$. Then $(\pi_{E_i})_{\an}\simeq\pf{a_{i+1},\dots,a_m}$. (For $i=m$ we get the $0$-fold Pfister form $\sqf{1}$.)
\end{lemma}

\begin{example} \label{Ex:fspPF}
Let $\pi\simeq\pf{a_1,\dots,a_n}$ be an anisotropic quasi-Pfister form. Since $(\pi_E)_{\an}$ is again a quasi-Pfister form (and hence its dimension is a $p$-power) for any field extension $E/F$ by \cite[Lemma~2.6]{Scu13}, it follows from Lemma~\ref{Lemma:PFinsepsplitting_pforms} that 
\[\fullsplitpat{\pi}=\{1,p,\dots,p^n\}.\]
\end{example}

%----------------------------------------------------------------------
\subsection{Full splitting pattern of quasi-Pfister neighbors}

We look at anisotropic $p$-forms of the form $\varphi\simeq\pi\ort d\sigma$, where $\pi$ is a quasi-Pfister form and $\sigma\subseteq\pi$. Note that $\varphi$ is a quasi-Pfister neighbor of $\pi\otimes\pf{d}$.

\begin{lemma}\label{Lemma:IsotropyIndicesSPN_pforms}
Let $\pi$ be a quasi-Pfister form over $F$, $\sigma\subseteq\pi$ and $d\in F^*$ be such that the $p$-form $\varphi\simeq\pi\ort d\sigma$ is anisotropic. Let $E/F$ be a field extension.
\begin{enumerate}
	\item If $d\in D_E(\pi)$, then we have $(\varphi_E)_{\an}\simeq(\pi_E)_{\an}$, and so, in particular, $\iql{\varphi_E}=\iql{\pi_E}+\dim\sigma$.
	\item If $d\notin D_E(\pi)$, then $(\varphi_E)_{\an}\simeq(\pi_E)_{\an}\ort d(\sigma_E)_{\an}$, and so, in particular, $\iql{\varphi_E}=\iql{\pi_E}+\iql{\sigma_E}$.
\end{enumerate}
\end{lemma}

\begin{proof}
If $d\in D_E(\pi)$, then $\varphi_E\simeq d\pi_E\ort d\sigma_E$, and hence $(\varphi_E)_{\an}\simeq(\pi_E)_{\an}$. 

Now assume $d\notin D_E(\pi)$; if $(\pi_E)_{\an}\ort d(\sigma_E)_{\an}$ were isotropic, then so would be $(\pi_E)_{\an}\otimes\sqf{1,d}$, which would imply $d\in D_E(\pi)$, a contradiction. Thus, the $p$-form $(\pi_E)_{\an}\ort d(\sigma_E)_{\an}$ is anisotropic. As clearly $(\varphi_E)_{\an}\simeq((\pi_E)_{\an}\ort d(\sigma_E)_{\an})_{\an}$, we get $(\varphi_E)_{\an}\simeq(\pi_E)_{\an}\ort d(\sigma_E)_{\an}$.
\end{proof}

\begin{corollary}\label{Cor:fspSPN_pforms}
Let $\pi$ be an $n$-fold quasi-Pfister form over $F$, $\sigma\subseteq\pi$ and $d\in F^*$ be such that the $p$-form $\varphi\simeq\pi\ort d\sigma$ is anisotropic. Then 
\[\fullsplitpat{\varphi}\subseteq\{p^k+\ell~|~0\leq k\leq n, 0\leq \ell\leq \dim\sigma\}.\]
\end{corollary}

\begin{lemma}\label{Lemma:SplitPatMinSubform_pforms}
Let $\pi\simeq\pf{a_1,\dots,a_n}$ be an anisotropic quasi-Pfister form over $F$ and $\sigma\simeq\sqf{1,a_1,\dots,a_s}$ for some $s\leq n$. Let $E/F$ be a field extension. 
\begin{enumerate}
	\item If $\dim(\sigma_E)_{\an}=\ell$, then $\dim(\pi_E)_{\an}\geq p^{\lceil\log_p\ell\rceil}$.
	\item If $\dim(\pi_E)_{\an}=p^k$, then $\dim(\sigma_E)_{\an}\geq k-n+s+1$.
\end{enumerate}
\end{lemma}

\begin{proof}
Part (i) follows from the facts that $(\sigma_E)_{\an}\subseteq(\pi_E)_{\an}$ and the dimension of $(\pi_E)_{\an}$ is a $p$-power.

To prove part (ii), we apply \cite[Prop.~5.2]{Hof04} on $\pi_E$ to find a subset $\{i_1,\dots,i_k\}\subseteq\{1,\dots,n\}$ such that $(\pi_E)_{\an}\simeq\pf{a_{i_1},\dots,a_{i_k}}_E$. Let 
\[\{j_1,\dots,j_\ell\}=\{i_1,\dots,i_k\}\cap\{1,\dots,s\};\] 
since the set $\{a_{j_1},\dots,a_{j_\ell}\}$ is $p$-independent over $E$, the $p$-form $\sqf{1,a_{j_1},\dots,a_{j_\ell}}_E$ is an anisotropic subform of $(\sigma_E)_{\an}$. In particular, $\dim(\sigma_E)_{\an}\geq \ell+1$. Since 
\[\ell=\abs{ \{i_1,\dots,i_k\}\cap\{1,\dots,s\}}\geq k-(n-s),\] 
the claim follows.
\end{proof}

If, in the situation of Lemma~\ref{Lemma:IsotropyIndicesSPN_pforms}, the form $\sigma$ is minimal, then it is possible to describe the full splitting pattern of $\varphi\simeq\pi\ort d\sigma$.

\begin{theorem}\label{Th:fullsplitpatSPNmin_pforms}
Let $\pi$ be an $n$-fold quasi-Pfister form over $F$, $\sigma$ be a minimal subform of $\pi$ of dimension at least $2$ and $d\in F^*$ be such that the $p$-form $\varphi\simeq\pi\ort d\sigma$ is anisotropic. Then $m\in \fullsplitpat{\varphi}$ if and only if $m=p^k+\ell$ and one of the following holds:
\begin{enumerate}
	\item $0\leq k\leq n$ and $\ell=0$;
	\item $0\leq k\leq n$ and $\max\{1, k-n+\dim\sigma\}\leq \ell\leq \min\{\dim\sigma,p^k\}$.
\end{enumerate}
\end{theorem}

\begin{proof}
If $x\in D_F(\sigma)$, then $x^{-1}\in G_F(\pi)$, and hence $x^{-1}\varphi\simeq\pi\ort d(x^{-1}\sigma)$ with $1\in D_F(x^{-1}\sigma)$; therefore, we can assume without loss of generality that $1\in D_F(\sigma)$. Let $\sigma\simeq\sqf{1,a_1,\dots, a_s}$ for some $1\leq s\leq n$; since $\sigma$ is assumed to be minimal, the set $\{a_1,\dots,a_s\}$ is $p$-independent by part (i) of Lemma~\ref{Lem:PropertiesOfMinimalForms_pforms}, and hence (by Lemma~\ref{Lemma:p-bases}) it can be extended to a $p$-independent set $\{a_1,\dots,a_n\}$ such that $N_F(\pi)=F^p(a_1,\dots, a_n)$. Then $\pi\simeq\pf{a_1,\dots,a_n}$ and
\[\varphi\simeq\pf{a_1\dots,a_n}\ort d\sqf{1,a_1,\dots,a_s}.\]

By Corollary~\ref{Cor:fspSPN_pforms}, we have
\[\fullsplitpat{\varphi}\subseteq\{p^k+\ell~|~0\leq k\leq n, 0\leq \ell\leq s+1\}.\]
Furthermore, if $\ell\geq 1$ (i.e., if $\sigma_E$ does not disappear completely in $(\varphi_E)_{\an}$), then by Lemma~\ref{Lemma:SplitPatMinSubform_pforms}, it must hold $k-n+s+1\leq \ell\leq p^k$ for any tuple $(k,\ell)$ such that $p^k+\ell\in\fullsplitpat{\varphi}$. Therefore, $\fullsplitpat{\varphi}\subseteq I$ where
\begin{multline*}
I=\{p^k+\ell~|~0\leq k\leq n, \max\{1, k-n+s+1\}\leq \ell\leq \min\{s+1,p^k\}\}\\
\cup\{p^k~|~0\leq k\leq n\}.
\end{multline*}
So we only need to prove that all values in $I$ are realizable. We define
\begin{align*}
	I_0&=\{(k,0)~|~0\leq k\leq n\},\\
	I_1&=\{(k,\ell)~|~0\leq k\leq n, \max\{1, k-n+s+1\}\leq \ell\leq\min\{s+1, k+1\}\},\\
	I_2&=\{(k,\ell)~|~0\leq k\leq n, k+1<\ell\leq \min\{s+1,p^k\}\};
\end{align*}
then $I=\{p^k+\ell~|~(k,\ell)\in I_0\cup I_1\cup I_2\}$.

If $(k,\ell)=(k,0)\in I_0$, then set $D_k=F(\sqrt[p]{a_{k+1}},\dots,\sqrt[p]{a_n}, \sqrt[p]{d})$. We get
\[(\varphi_{D_k})_{\an}\simeq\pf{a_1,\dots,a_k}_{D_k},\]
i.e., $\dim(\varphi_{D_k})_{\an}=p^k$.

To cover the values in the set $I_1$, we define $E_{n,s+1}=F$ and 
\[E_{k,\ell}=F(\sqrt[p]{a_\ell},\dots,\sqrt[p]{a_{n-(k-\ell)-1}})\]
for all tuples $(k,\ell)\in I_1$ such that $(k,\ell)\neq (n,s+1)$; note that since $k\leq n-1$ for such tuples, we have $\ell\leq n-(k-\ell)-1$. Moreover, the condition $\ell\leq k+1$ can be rewritten as $n-(k-\ell)-1 \leq n$. Therefore, $E_{k,\ell}$ is well-defined. Moreover, the inequality $k-n+s+1\leq \ell$ ensures that $s\leq n-(k-\ell)-1$, and hence $(\sigma_{E_{k,\ell}})_\an\simeq\sqf{1,a_1,\dots,a_{\ell-1}}_{E_{k,\ell}}$.
Thus, for any $(k,\ell)\in I_1$, we get 
\[(\varphi_{E_{k,\ell}})_{\an}\simeq(\pf{a_1,\dots,a_{\ell-1}}\otimes\pf{a_{n-(k-\ell)},\dots,a_n}\ort d\sqf{1,a_1,\dots,a_{\ell-1}})_{E_{k,\ell}}\]
and so $\dim(\varphi_{E_{k,\ell}})_{\an}=p^k+\ell$.

Finally, let $(k,\ell)\in I_2$. We denote $S_k=\{\lambda:\{1,\dots,k\}\rightarrow\{0,\dots,p-1\}\}$ and write $a^\lambda=\prod_{i=1}^ka_i^{\lambda(i)}$ for any $\lambda\in S_k$. Furthermore, we pick a subset $\mathcal{L}_{k,\ell}\subseteq S_k$ such that $\abs{\mathcal{L}_{k,l}}=\ell-k-1$ and $\sum_{i=1}^k\lambda(i)>1$ for each $\lambda\in\mathcal{L}_{k,\ell}$; note that this is possible since $\ell-k-1\leq p^k-k-1=\abs{\{\lambda\in S_k~|~\sum_{i=1}^k\lambda(i)>1\}}$. Let $\mathcal{L}_{k,\ell}=\{\lambda_1,\dots,\lambda_{\ell-k-1}\}$. We define
\[G_{k,\ell}=F\left(\sqrt[p]{\frac{a^{\lambda_1}}{a_{k+1}}}, \dots, \sqrt[p]{\frac{a^{\lambda_{\ell-k-1}}}{a_{\ell-1}}}, \sqrt[p]{a_\ell},\dots,\sqrt[p]{a_n}\right)\]
(the field depends on the choice of $\mathcal{L}_{k,\ell}$ and the ordering of its elements). Then $[G_{k,\ell}:F]\leq p^{n-k}$, and
\[
G_{k,\ell}^p(a_1,\dots,a_k)
=F^p\left(a_1,\dots,a_k, \frac{a^{\lambda_1}}{a_{k+1}}, \dots, \frac{a^{\lambda_{\ell-k-1}}}{a_{\ell-1}}, a_\ell, \dots,a_n\right)
=F^p(a_1,\dots,a_n)
\]
because $a^{\lambda_j}\in F^p(a_1,\dots,a_k)$ for all $1\leq j\leq \ell-k-1$. As $[G^p_{k,\ell}:F^p]=[G_{k,\ell}:F]$ by Lemma~\ref{Lemma:AlgebraicToThepPower}, we get
\[p^n=[F^p(a_1,\dots,a_n):F^p]=[G_{k,\ell}^p(a_1,\dots,a_k):G_{k,\ell}^p]\cdot[G^p_{k,\ell}:F^p]\leq p^k\cdot p^{n-k}.\]
It follows that $[G_{k,\ell}:F]=p^{n-k}$ and $[G_{k,\ell}^p(a_1,\dots,a_k):G_{k,\ell}^p]=p^k$; hence, in particular, the set $\{a_1,\dots,a_k\}$ is $p$-independent over $G_{k,\ell}$, which means that the $p$-form $\pf{a_1,\dots,a_k}$ is anisotropic over $G_{k,\ell}$. Moreover, note that we have $a^{\lambda_j}\equiv a_{k+j} \mod G_{k,\ell}^{*p}$ for any $1\leq j\leq \ell-k-1$; in particular $a_{k+1},\dots,a_{\ell-1}\in D_{G_{k,\ell}}(\pf{a_1,\dots,a_k})$. Therefore, 
\begin{align*}
(\pi_{G_{k,\ell}})_{\an}&\simeq\pf{a_1,\dots,a_k}_{G_{k,\ell}},\\
(\sigma_{G_{k,\ell}})_{\an}&\simeq\sqf{1,a_1,\dots,a_{\ell-1}}_{G_{k,\ell}},
\end{align*}
where the anisotropy of $\sqf{1,a_1,\dots,a_{\ell-1}}_{G_{k,\ell}}$ follows from
\[\sqf{1,a_1,\dots,a_{\ell-1}}_{G_{k,\ell}}\simeq\bigl(\sqf{1,a_1,\dots,a_k}\ort\Ort\limits_{j=1}^{\ell-k-1}\sqf{a^{\lambda_j}}\bigr)_{G_{k,\ell}}\subseteq\pf{a_1,\dots,a_k}_{G_{k,\ell}}.\]
Since $d\notin F^p(a_1,\dots,a_n)=D_{G_{k,\ell}}(\pi)$, we get by Lemma~\ref{Lemma:IsotropyIndicesSPN_pforms} that
\[(\varphi_{G_{k,\ell}})_{\an}\simeq\left(\pf{a_1,\dots,a_k}\ort d\sqf{1,a_1,\dots,a_{\ell-1}}\right)_{G_{k,\ell}}; \]
in particular $\dim(\varphi_{G_{k,\ell}})_{\an}=p^k+\ell$.
\end{proof}

\begin{remark}
Note that we used in the proof of Theorem~\ref{Th:fullsplitpatSPNmin_pforms} only purely inseparable field extensions of exponent one. Thus, for any $p$-form $\varphi$ satisfying the conditions of that theorem, we have $\pisp{\varphi}=\fsp{\varphi}$.
\end{remark}

The following example illustrates both the result and the proof of Theorem~\ref{Th:fullsplitpatSPNmin_pforms}.

\begin{example} \label{Ex:FullSplitPattSPN_pforms}
Let $a_1, a_2, a_3, a_4, d\in F^*$ be $p$-independent and
\[\varphi\simeq\pf{a_1,a_2,a_3, a_4}\ort d\sqf{1,a_1,a_2, a_3}.\]
Then we have by Theorem~\ref{Th:fullsplitpatSPNmin_pforms}
\begin{align*}
\fullsplitpat{\varphi}=&\{p^0, p^1,p^2,p^3,p^4\}\\
&\cup\{p^0+1,p^1+1,p^1+2,p^2+2,p^2+3,p^2+4, p^3+3,p^3+4\}\\
&\cup\{p^1+3 ~|~\text{if } p\geq3\}\cup\{p^1+4~|~\text{if } p\geq4\}.
\end{align*}
Table~\ref{Tab:FullSplitPattSPN_pforms} provides the fields used in the proof of Theorem~\ref{Th:fullsplitpatSPNmin_pforms} and the obtained $p$-forms for the case $p\geq4$.
\end{example}

\newgeometry{left=3cm,right=2cm, bottom=2cm, top=2cm}
\begin{landscape}
\begin{table}%
\begin{adjustbox}{max width=24.5cm}
\renewcommand{\arraystretch}{2}
\begin{tabular}{|c||c|c|c|c|c|}
\hline
$\dim$&0&1&2&3&4\\ \hline\hline
\multirow{2}{*}{$p^0$}
		& $D_0=F(\sqrt[p]{a_1},\sqrt[p]{a_2},\sqrt[p]{a_3},\sqrt[p]{a_4},\sqrt[p]{d})$ & $E_{0,1}=F(\sqrt[p]{a_1},\sqrt[p]{a_2},\sqrt[p]{a_3},\sqrt[p]{a_4})$ & \multirow{2}{*}{X} & \multirow{2}{*}{X} & \multirow{2}{*}{X}\\ 
		&	$\sqf{1}$ &$\sqf{1}\ort d\sqf{1}$&&&\\ \hline
\multirow{2}{*}{$p^1$}
		& $D_1=F(\sqrt[p]{a_1},\sqrt[p]{a_2},\sqrt[p]{a_3},\sqrt[p]{d})$ & $E_{1,1}=F(\sqrt[p]{a_1},\sqrt[p]{a_2},\sqrt[p]{a_3})$ & $E_{1,2}=F(\sqrt[p]{a_2},\sqrt[p]{a_3},\sqrt[p]{a_4})$ & $G_{1,3}=F\bigl(\sqrt[p]{\frac{a_1^2}{a_2}},\sqrt[p]{a_3},\sqrt[p]{a_4}\bigr)$  & $G_{1,4}=F\bigl(\sqrt[p]{\frac{a_1^2}{a_2}}, \sqrt[p]{\frac{a_1^3}{a_3}}, \sqrt[p]{a_4}\bigr)$ \\ 
		& $\pf{a_4}$ &$\pf{a_4}\ort d\sqf{1}$&$\pf{a_1}\ort d\sqf{1,a_1}$ & $\pf{a_1}\ort d\sqf{1,a_1,a_1^2}$ & $\pf{a_1}\ort d\sqf{1,a_1,a_1^2, a_1^3}$ \\ \hline
\multirow{2}{*}{$p^2$}
		& $D_2=F(\sqrt[p]{a_1},\sqrt[p]{a_2},\sqrt[p]{d})$ & \multirow{2}{*}{X} & $E_{2,2}=F(\sqrt[p]{a_2},\sqrt[p]{a_3})$ & $E_{2,3}=F(\sqrt[p]{a_3},\sqrt[p]{a_4})$ & $G_{2,4}=F\bigl(\sqrt[p]{\frac{a_1a_2}{a_3}}, \sqrt[p]{a_4}\bigr)$ \\ 
		& $\pf{a_3,a_4}$ & & $\pf{a_1,a_4}\ort d\sqf{1,a_1}$ & $\pf{a_1,a_2}\ort d\sqf{1,a_1,a_2}$ & $\pf{a_1,a_2}\ort d\sqf{1,a_1,a_2,a_1a_2}$\\ \hline
\multirow{2}{*}{$p^3$}
		& $D_3=F(\sqrt[p]{a_1},\sqrt[p]{d})$ & \multirow{2}{*}{X} & \multirow{2}{*}{X} & $E_{3,3}=F(\sqrt[p]{a_3})$ & $E_{3,4}=F(\sqrt[p]{a_4})$ \\ 
		& $\pf{a_2,a_3,a_4}$ & & & $\pf{a_1,a_2,a_3}\ort d\sqf{1,a_1,a_2}$ & $\pf{a_1,a_2,a_3}\ort d\sqf{1,a_1,a_2,a_3}$\\ \hline
\multirow{2}{*}{$p^4$}
		& $D_4=F(\sqrt[p]{d})$ & \multirow{2}{*}{X} & \multirow{2}{*}{X} & \multirow{2}{*}{X} & $E_{4,4}=F$ \\ 
		& $\pf{a_1,a_2,a_3,a_4}$ & & & & $\pf{a_1,a_2,a_3,a_4}\ort d\sqf{1,a_1,a_2,a_3}$\\ \hline
\end{tabular}
\end{adjustbox}
\vspace{4mm}
\caption{Splitting of the $p$-form $\pf{a_1,a_2,a_3,a_4}\ort d\sqf{1,a_1,a_2,a_3}$ in the case $p\geq4$ (see Example~\ref{Ex:FullSplitPattSPN_pforms}).}
\label{Tab:FullSplitPattSPN_pforms}
\end{table}
\end{landscape}
\restoregeometry
%---------------------------------------------------
\newpage
\printbibliography

\end{document}